\documentclass[11pt]{article}
\usepackage[pagebackref=true,colorlinks,linkcolor=green,citecolor=magenta]{hyperref}
\usepackage{amssymb,amsmath}
\usepackage{amsmath}
\usepackage{dsfont}
\usepackage{amsfonts}
\newtheorem{precor}{{\bf Corollary}}

\newtheorem{precon}{{\bf Conjecture}}

\newtheorem{prealphcon}{{\bf Conjecture}}

\newtheorem{predefin}{{\bf Definition}}

\newenvironment{defin}[1]{\begin{predefin}{\hspace{-0.5
                   em}{\bf.\ }}{\rm #1}\hfill{$\spadesuit$}}{\end{predefin}}
\newtheorem{preexm}{{\bf Example}}

\newtheorem{preappl}{{\bf Application}}

\newtheorem{prelem}{{\bf Lemma}}

\newenvironment{lem}{\begin{prelem}{\hspace{-0.5
               em}{\bf.\ }}}{\end{prelem}}
\newtheorem{preproof}{{\bf Proof.\ }}

\newenvironment{proof}[1]{\begin{preproof}{\rm
               #1}\hfill{$\blacksquare$}}{\end{preproof}}
\newtheorem{prethm}{{\bf Theorem}}

\newenvironment{thm}{\begin{prethm}{\hspace{-0.5
               em}{\bf.\ }}}{\end{prethm}}
\newtheorem{prealphthm}{{\bf Theorem}}

\newenvironment{alphthm}{\begin{prealphthm}{\hspace{-0.5
               em}{\bf.\ }}}{\end{prealphthm}}

\newtheorem{prealphlem}{{\bf Lemma}}

\newenvironment{alphlem}{\begin{prealphlem}{\hspace{-0.5
               em}{\bf.\ }}}{\end{prealphlem}}

\newtheorem{prepro}{{\bf Proposition}}

\newtheorem{preprb}{{\bf Problem}}

\newtheorem{prerem}{{\bf Remark}}

\newtheorem{preapp}{{\bf Application}}

\newtheorem{prequ}{{\bf Question}}

\newtheorem{preclaim}{{\bf Claim}}

\def\conct[#1,#2]{\mbox {${#1} \leftrightarrow {#2}$}}
\def\dconct[#1,#2]{\mbox {${#1} \rightarrow {#2}$}}
\def\deg[#1,#2]{\mbox {$d_{_{#1}}(#2)$}}
\def\mindeg[#1]{\mbox {$\delta_{_{#1}}$}}
\def\maxdeg[#1]{\mbox {$\Delta_{_{#1}}$}}
\def\outdeg[#1,#2]{\mbox {$d_{_{#1}}^{^+}(#2)$}}
\def\minoutdeg[#1]{\mbox {$\delta_{_{#1}}^{^+}$}}
\def\maxoutdeg[#1]{\mbox {$\Delta_{_{#1}}^{^+}$}}
\def\indeg[#1,#2]{\mbox {$d_{_{#1}}^{^-}(#2)$}}
\def\minindeg[#1]{\mbox {$\delta_{_{#1}}^{^-}$}}
\def\maxindeg[#1]{\mbox {$\Delta_{_{#1}}^{^-}$}}

\def\dre[#1,#2,#3]{\mbox {${\cal E}^{^{#3}}(#1,#2)$}}
\def\var[#1,#2]{\mbox {${\rm Var}_{_{#1}}(#2)$}}
\def\ls[#1]{\mbox {$\xi^{^{#1}}$}}
\def\hom[#1,#2]{\mbox {${\rm Hom}({#1},{#2})$}}
\def\onvhom[#1,#2]{\mbox {${\rm Hom^{v}}(#1,#2)$}}
\def\onehom[#1,#2]{\mbox {${\rm Hom^{e}}(#1,#2)$}}
\def\core[#1]{\mbox {$#1^{^{\bullet}}$}}
\def\cay[#1,#2]{\mbox {${\rm Cay}({#1},{#2})$}}
\def\sch[#1,#2,#3]{\mbox {${\rm Sch}({#1},{#2},{#3})$}}
\def\cays[#1,#2]{\mbox {${\rm Cay_{s}}({#1},{#2})$}}
\def\dirc[#1]{\mbox {$\stackrel{\rightarrow}{C}_{_{#1}}$}}
\def\cycl[#1]{\mbox {${\bf Z}_{_{#1}}$}}

\begin{document}
\begin{center}
{\Large \bf  On Chromatic Number and Minimum Cut}\\
\vspace{0.3 cm}
{\bf Meysam Alishahi$^\dag$ and Hossein Hajiabolhassan$^\ast$\\
{\it $^\dag$ School of Mathematical Sciences}\\
{\it University of Shahrood, Shahrood, Iran}\\
{\tt meysam\_alishahi@shahroodut.ac.ir}\\
{\it $^\ast$ Department of Applied Mathematics and Computer Science}\\
{\it Technical University of Denmark}\\
{\it DK-{\rm 2800} Lyngby, Denmark}\\
{\it $^\ast$ Department of Mathematical Sciences}\\
{\it Shahid Beheshti University, G.C.}\\
{\it P.O. Box {\rm 19839-63113}, Tehran, Iran}\\
{\tt hhaji@sbu.ac.ir}\\
}
\end{center}
\begin{abstract}
\noindent 
For a graph $G$,  the tree graph ${\cal T}_{G,t}$ has all tree subgraphs of $G$ with $t$ vertices as vertex set and two tree subgraphs are neighbors if they are edge-disjoint. Also, the $r^{th}$ cut number of $G$ is the minimum number of edges between parts of a partition of vertex set of $G$ into two parts such that each part has size at least $r$. We show that 
if $t=(1-o(1))n$ and $n$ is large enough, then for any dense graph $G$ with $n$ vertices, the chromatic number of  the tree graph ${\cal T}_{G,t}$ is equal to  the $(n-t+1)^{th}$ cut number of $G$. In particular, as a consequence, we prove that if $n$ is large enough and $G$ is a dense graph, then the chromatic number of the spanning tree graph ${\cal T}_{G,n}$ is equal to the size of the minimum cut of $G$. The proof method is based on alternating Tur\'an number inspired by Tucker's lemma, an equivalent combinatorial version of the Borsuk-Ulam theorem.\\

\end{abstract}
\section{Introduction}
 The usual {\it Kneser graph} ${\rm KG(n,k)}$ is a graph whose vertex set consists of all $k$-subsets of $[n]=\{1,\ldots,n\}$ and two vertices are adjacent if the corresponding $k$-subsets are disjoint. 
The usual Kneser graphs are generalized in several ways. For a given graph $G$ and a positive integer $t$, where  $1\leq t\leq |V(G)|$,  
define the {\it tree graph} ${\cal T}_{G,t}$ to be a graph whose vertex set consists of all tree subgraphs of $G$ with $t$ vertices and two vertices are adjacent if the corresponding tree subgraphs are edge-disjoint. 
Note that ${\cal T}_{G,t}$ can be considered as a generalization of the usual Kneser graphs.
For this purpose, note that if $G$ is a star with $n$ edges, then the graph ${\cal T}_{G,t}$ is isomorphic  to ${\rm KG}(n,t)$. 
 
A partition $(X,Y)$ of the vertex set of a graph $G$ is called a {\it cut of $G$}.
 The size of this cut is the number of edges in $G$ which meet both $X$ and $Y$, that is, 
 $|E(G[X,Y])|$. A cut $(X,Y)$ is called an $r$-cut, if $\min\{|X|,|Y|\}\geq r$.
The minimum possible size of an $r$-cut in $G$ is called the $r^{th}$ cut number of $G$ and is denoted by ${\rm cut}_r(G)$,
 
 In this paper, we are interested in finding the chromatic number of the tree graph ${\cal T}_{G,t}$. 
Let $t\in\{1,2,\ldots,\lfloor{n\over 2}\rfloor-1\}$, where $|V(G)|=n$. Since the vertex set of any tree subgraph of $G$ with $n-t$ vertices meets both sides of any $(t+1)$-cut of $G$, we have  $$\chi({\cal T}_{G,n-t})\leq {\rm cut}_{t+1}(G).$$
As a main result of this paper, we prove the next theorem.
\begin{thm}\label{firstthm}
Let $n$ and $r$ be nonnegative integers and $\delta$ be a real number, where 
${5\over 6}< \delta < 1$ and $r=o(n)$.
If $n$ is sufficiently large, then
for any graph $G$ with $n$ vertices and $\delta(G)\geq \delta n$ we have 
$$\chi({\cal T}_{G,n-r})={\rm cut}_{r+1}(G).$$
\end{thm}
It should be mentioned that we actually  prove a more general statement (Theorem~\ref{mainthm}) which implies the previous theorem immediately. However, we need some technical definitions to state this statement.

\section{Motivation and Background}
A graph $G$ is called a {\it labeled graph}, if there exists a bijective map from 
the vertex set of $G$ to a set of labels of size $|V(G)|$. Two labeled graphs  $G$ and $H$ are isomorphic, if 
there is a bijective map between their vertex sets which preserves adjacency, 
non-adjacency, and labels. Also, 
for any subgraph $H$ of $G$, the labels of the vertices of $H$ are inherited from $G$. 
For an unlabeled (resp. labeled) graph $G$  and a nonempty family ${\cal F}$ of unlabeled (resp. labeled) graphs, 
the general Kneser graph ${\rm KG}(G, {\cal F})$ has all subgraphs of $G$ isomorphic to some members of ${\cal F}$ as vertex set and two vertices of ${\rm KG}(G, {\cal F})$ are adjacent 
if the corresponding subgraphs are edge-disjoint. 
It is known that for any graph $G$ there are many general Kneser graphs isomorphic to $G$, 
see~\cite{2013arXiv1302.5394A, 2014arXiv1401.0138A}. 
In~\cite{MR2117355}, it was shown that if the complement of $G$ is a connected graph, then there is a labeled tree $T$ and a family ${\cal T}$ of subtrees $T$ such that $G$ and ${\rm KG}(T, {\cal T})$ are isomorphic. Also, it is known that if $T_3$ is a labeled tree with maximum degree at most $3$ and ${\cal T}$ is a family of 
subtrees of $T_3$, then the graph 
${\rm KG}(T, {\cal T})$
is isomorphic to the complement of a chordal graph, and moreover, the complement of any choral graph has
such a representation, see~\cite{MR2117355}. For more about general Kneser graphs, see~\cite{2013arXiv1302.5394A, 2014arXiv1401.0138A, 2014arXiv1403.4404A}.

Hereafter, unless otherwise stated, we only consider unlabeled graphs. Several well-known families of graphs can be represented by some general Kneser graphs of unlabeled graphs. For instance, the general Kneser graphs ${\rm KG}(nK_2, kK_2)\cong{\rm KG}(K_{1,n}, K_{1,k})$, ${\rm KG}(C_n, kK_2)$, and ${\rm KG}(C_n, P_{d})$,  
are isomorphic to the Kneser graph ${\rm KG}(n, k)$, the Schrijver graph ${\rm SG}(n, k)$ and the circular complete graph $K_{n\over d}$, respectively, where $P_{d}$ is a path of length $d$.  

In this paper, we investigate the chromatic number ${\rm KG}(G,{\cal T}_t)\cong {\cal T}_{G,t}$, where 
the family ${\cal T}_t$ consists of all tree subgraphs of $G$ with $t$ vertices.
This problem is motivated by the problem of finding the clique number of spanning tree graphs, i.e., the maximum number of edge-disjoint spanning tree subgraphs of a graph. The later problem has been studied in the literature, see~\cite{MR0369117, MR1812338}. It was shown \cite{MR0369117} that  the maximum number of edge-disjoint spanning tree graphs of a graph $G$ is bounded below by $\lfloor{{\rm cut}_1(G)-1\over 2}\rfloor$.
It is simple to see that the maximum number of edge-disjoint spanning tree graphs of $G$ is bounded above by the size of the minimum cut of $G$. In this paper, we show that the chromatic of the spanning tree graph of a dense graph $G$ with a  large number of vertices is equal to the size of the minimum cut (the minimum degree) of the graph $G$. 

This paper is organized as follows. Section~\ref{altdef} is concerned with 
some notation that will be used later. In particular, we introduce the concept of alternating Tur\'an number. Section~\ref{resultsec} is devoted to the chromatic number of tree graphs. 
\section{Alternating Tur\'an Number}\label{altdef}

Let $G$ be a graph, $H$ be a subgraph of $G$, and ${\cal G}=\{G_1,\ldots,G_k\}$ be a family of nonempty subgraphs of $G$. The subgraph 
$H$ is called  a {\it  ${\cal G}$-subgraph}, if for any $1\leq i \leq k$, $E(H)\cap E(G_i)\not=~\varnothing$. 
In this regard, by a {\it ${\cal G}$-forest {\rm(}{\it ${\cal G}$-tree}{\rm)}}, we mean 
a forest (resp. tree) which is also a ${\cal G}$-subgraph. 
Also, throughout the paper, we write $G\setminus H$ and $G-H$ for the subgraphs of $G$ obtained by 
removing the edges and the vertices of $H$, respectively. 
For any family of graphs ${\cal F}$, the {\it generalized Tur\'an number} ${\rm ex}(G,{\cal G}, {\cal F})$ is the maximum number of edges of a spanning subgraph of 
$G$ such that it has no ${\cal G}$-subgraph isomorphic to some member of ${\cal F}$. 
For a graph $G$ and a linear ordering $\sigma: e_1<e_2<\cdots<e_{|E(G)|}$ of  $E(G)$, 
an {\it alternating  $2$-coloring} of $E(G)$ with respect to the 
ordering $\sigma$ is a mapping which assigns alternatively (with respect to $\sigma$) two colors red and blue to a subset of $E(G)$.  
For an edge $e\in E(G)$, if the red (resp. blue) color is assigned to $e$, then this edge is called a {\it red edge} (resp. {\it blue edge}) and if no color is assigned to $e$, then it is called a {\it neutral} edge.
Note that 
in an alternating $2$-coloring of $E(G)$,
any two consecutive colored edges (with respect to the ordering $\sigma$) have different colors.
Moreover, the  {\it red  subgraph} $G^R$ (resp.  {\it blue  subgraph} $G^B$) is a spanning subgraph of $G$ whose edge set is the set of all red (resp. blue) edges of $G$. 
The {\it length} of an alternating $2$-coloring of $E(G)$ is the number of colored edges, i.e., $|E(G^R)|+ |E(G^B)|$. 
For a subgraph $H$ of $G$, 
the set of neutral edges of $E(H)$ is denoted by ${\rm NEU}(H)$. Moreover, 
the set of neutral edges of $E(H)$ incident with a vertex $v$ of $H$ is denoted by ${\rm NEU}(v,H)$. 
Define ${\rm ex}_{alt}(G,{\cal G}, {\cal F}, \sigma)$ to be the maximum length of an alternating $2$-coloring $h$ of $E(G)$ 
with respect to the  
ordering $\sigma$ such that it has no monochromatic (red or blue) ${\cal G}$-subgraph isomorphic to a member of ${\cal F}$, or equivalently,  each of $G^R$ and $G^B$ has no ${\cal G}$-subgraph isomorphic to a member of ${\cal F}$. 
Define the {\it generalized alternating Tur\'an number} ${\rm ex}_{alt}(G,{\cal G},{\cal F})$ as follows
$${\rm ex}_{alt}(G,{\cal G},{\cal F})=\displaystyle\min_{\sigma} \left\{{\rm ex}_{alt}(G,{\cal G},{\cal F},\sigma): \sigma\ {\rm is}\  {\rm an}\  {\rm ordering}\  {\rm of}\ E(G)\right\}.$$
In view of definition of generalized Tur\'an number and generalized alternating Tur\'an number, one can check that 
\begin{equation}\label{exexalt}
{\rm ex}(G, {\cal G},{\cal F})\leq {\rm ex}_{alt}(G, {\cal G}, {\cal F}) \leq 2{\rm ex}(G, {\cal G},{\cal F}).
\end{equation} 
The general Kneser graph ${\rm KG}(G, {\cal G}, {\cal F})$ has all ${\cal G}$-subgraphs of $G$ isomorphic to some members of 
${\cal F}$ as vertex set and two vertices are adjacent if the corresponding ${\cal G}$-subgraphs are edge-disjoint. Note that ${\rm KG}(G, {\cal G}, {\cal F})$ is a subgraph of ${\rm KG}(G, {\cal F})$ and if ${\cal G}=\{G\}$, then two graphs ${\rm KG}(G, {\cal G}, {\cal F})$ and ${\rm KG}(G, {\cal F})$ are isomorphic. 
In~\cite{2014arXiv1401.0138A}, several examples presented to show that equality holds in both of inequalities 
in~(\ref{exexalt}). One can see that the chromatic number of ${\rm KG}(G, {\cal G}, {\cal F})$ is bounded above by 
$|E(G)|-{\rm ex}(G,{\cal G}, {\cal F})$. To see this, consider a subgraph $H$ of $G$ with ${\rm ex}(G,{\cal G}, {\cal F})$ edges such that it has no ${\cal G}$-subgraph isomorphic to any member of ${\cal F}$. Assume that $E(G)\setminus E(H)=\{e_1,\ldots,e_t\}$. 
For any vertex of ${\rm KG}(G, {\cal G}, {\cal F})$, which is a ${\cal G}$-subgraph of $G$, consider the minimum number $i$ such that $e_i$ is 
an edge of this ${\cal G}$-subgraph. Assign $i$ as a color to this vertex. 
One can see that it is a proper coloring for ${\rm KG}(G, {\cal G}, {\cal F})$. 
On the other hand, it is proved that the chromatic number of ${\rm KG}(G, {\cal G}, {\cal F})$ is bounded below by $|E(G)|-{\rm ex}_{alt}(G, {\cal G}, {\cal F})$ \cite{2013arXiv1302.5394A,2014arXiv1401.0138A}. 
\begin{alphlem}\label{ex}{\rm \cite{2013arXiv1302.5394A,2014arXiv1401.0138A}}
Assume that $G$ is a graph and ${\cal G}=\{G_1,\ldots,G_k\}$ is a family of subgraphs of $G$.  
For any family ${\cal F}$ of graphs, we have 
$$\displaystyle
|E(G)|-{\rm ex}_{alt}(G, {\cal G}, {\cal F}) \leq \chi({\rm KG}(G, {\cal G}, {\cal F})) \leq  |E(G)|-{\rm ex}(G, {\cal G}, {\cal F}).$$
\end{alphlem}

It should be mentioned that the lower bound presented in Lemma~\ref{ex} was originally proved in~\cite{2013arXiv1302.5394A}, where it was stated in different terms. However, it was restated in terms of alternating Tur\'an numbers in~\cite{2014arXiv1401.0138A}.
\section{Tree Graphs}\label{resultsec}
Let $G$ be a graph and ${\cal T}_t$ be the family of all tree subgraphs of $G$ with $t$ vertices.
In this section, we investigate the chromatic number of the tree graph ${\rm KG}(G,{\cal T}_t)$ provided that $G$ is dense graph with sufficiently large number of vertices. In~\cite{MR0369117}, it was shown  
the clique number of the spanning tree graph ${\rm KG}(G,{\cal T}_n)$ is bounded below by $\lfloor{{\rm cut}_1(G)-1\over 2}\rfloor$, 
where ${\rm cut}_1(G)$ is the size of the minimum cut of the graph $G$. We show that if 
$n$ is sufficiently large, then for any dense graph $G$ with $n$ vertices, 
the chromatic number of the tree graph ${\rm KG}(G,{\cal T}_n)$ is equal to the size of the minimum cut of $G$. Note that, in general, equality does~not hold. For instance, $\chi({\rm KG}(K_3,{\cal T}_3))=1$ and $\chi({\rm KG}(K_4,{\cal T}_4))=2$. 

For two graphs $G$ and $H$, an {\it $H$-packing} of $G$ is a set $\{H_1,\ldots,H_t\}$ of pairwise edge-disjoint subgraphs of $G$ such that for each $1 \leq i \leq t$, the graph $H_i$ is isomorphic to $H$. 
Moreover, if the edge sets of the $H_i$'s partition the edge set of $G$, then it is termed an {\it $H$-decomposition}. 
The $H$-packing number of $G$ is the maximum cardinality of an $H$-packing of $G$. An $H$-packing (resp. $H$-decomposition) of $G$  is called a {\it monogamous $H$-packing} (resp. {\it monogamous $H$-decomposition}), if every pair of vertices of $G$ appears in at most one copy of $H$ in the packing (resp. decomposition). In~\cite{MR1723879}, it was shown that for positive even integers 
$m$ and $n$, the complete bipartite graph $K_{m,n}$ has a  monogamous $C_4$-decomposition if and only if $(m,n)=(2,2)$ or $6\leq n\leq m\leq 2n-2$. Note that if a graph $G$ has a $K_t$-packing, then it is a monogamous $K_t$-packing. 
The problem of finding as many as possible vertex-disjoint complete subgraphs of specified order
has been studied in several articles. Hajnal and Szemer{\'e}di~\cite{MR0297607} proved the next theorem. 
\begin{alphthm}{\rm \cite{MR0297607}}\label{Hajnal}
Assume that $G$ is a graph with $n$ vertices.
If $\delta(G) \geq (1-{1\over r})n$, then $G$ contains $\lfloor{n\over r}\rfloor$ 
vertex-disjoint copies of $K_r$.
\end{alphthm} 
Let $G$ be a graph with $n$ vertices and ${\cal G}=\{G_1,\ldots,G_k\}$ be a family of pairwise edge-disjoint subgraphs of $G$.
For any integer $1\leq i\leq n-1$, define the {\it $i^{th}$ cut number} of $G$ with respect to ${\cal G}$, ${\rm cut}_i(G, {\cal G})$,  as follows
$${\rm cut}_i(G, {\cal G})=|E(G)|-{\rm ex}(G, {\cal G}, {\cal T}_{n-i+1}).$$
If ${\cal G}=\{G\}$, then set ${\rm cut}_i(G)={\rm cut}_i(G, {\cal G})$. One can readily check that if  
$1\leq i \leq \lfloor {n\over 2}\rfloor$, then
$${\rm cut}_i(G)=\min\left\{|E(S,V(G)\setminus S)|:\ S\subseteq V(G),\ \ |S|\geq i,\ |V(G)\setminus S|\geq i \right\}.$$
Note that ${\rm cut}_1(G)$ is the size of the minimum cut of $G$. 

Consider a connected graph $G$ with $n$ vertices. Also, let ${\cal G}=\{G_1,G_2,\ldots,G_k\}$ 
be a decomposition of $G$. 
Lemma~\ref{ex} gives a lower and an upper bound for the chromatic number of ${\rm KG}(G,{\cal G}, {\cal T}_t)$ in terms of $|E(G)|$, ${\rm ex}(G, {\cal G}, {\cal T}_{t})$, and ${\rm ex}_{alt}(G, {\cal G}, {\cal T}_{t})$. One can see that 
the upper bound is equal to ${\rm cut}_{n-t+1}(G, {\cal G})$ provided that $t\geq \lceil{n\over 2}\rceil+1$.  
It would be of interest to find some sufficient conditions which make this upper bound sharp.
Next theorem, which is the main result of this paper, provides some sufficient conditions 
for the equality of  ${\rm cut}_{n-t+1}(G, {\cal G})$ and the chromatic number of ${\rm KG}(G,{\cal G}, {\cal T}_t)$. 

\begin{thm}\label{mainthm}
Let $k, n,$ and $r$ be positive integers and $\delta$ be a real number, where 
${5\over 6}< \delta < 1$, $k=o(\sqrt{n})$, and $r=o(n)$. There is a constant integer
$N=N(k,\delta,r)$ such that for any graph $G$ with $n\geq N$ vertices and any 
decomposition ${\cal G}=\{G_1,G_2,\ldots,G_k\}$  of  $G$ satisfying the following conditions
\begin{itemize}
\item $1\leq |E(G_1)|\leq |E(G_{2})|\leq \cdots\leq |E(G_k)|$,
\item either $k\leq 2$ or $\sum_{i=1}^3|E(G_i)|\geq 3{2k-3\choose 2}+4$, 
\item $G_k$ is a spanning subgraph of $G$ with $\delta(G_k)\geq \delta n$, 
\end{itemize}
we have
$$\chi({\rm KG}(G, {\cal G}, {\cal T}_{ n-r}))={\rm cut}_{r+1}(G, {\cal G}).$$
\end{thm}
The rest of the paper is devoted to the proof of the previous theorem. 
\begin{defin}{
Let $G$ be a graph and ${\cal G}=\{G_1,G_2,\ldots,G_k\}$ be a decomposition  of $G$, 
where $1\leq|E(G_1)|\leq |E(G_2)|\leq \cdots \leq |E(G_k)|$. 
For an ordering $\sigma \in S_{E(G)}$,
we say that $G$ has {\it $({\cal G},\sigma)$-forest property} if for any alternating $2$-coloring 
of $E(G)$  of length at least $|E(G)|-|E(G_1)|+1$ with respect to the ordering $\sigma$, 
the following property holds. 
If $G^R$ (resp.  $G^B$) is a connected spanning ${\cal G}$-subgraph, then it has a  ${\cal G}$-forest.
Moreover, $G$ has {\it ${\cal G}$-forest property} if for any ordering $\sigma \in S_{E(G)}$
in which the edges of each $G_i$ are consecutive, i.e., $\sigma=\sigma_1||\sigma_2||\cdots||\sigma_k$ and $\sigma_i \in S_{E(G_{j_i})}$ is an ordering of 
the edge set of $G_{j_i}$ for $1\leq i\leq k$,  
$G$ has $({\cal G},\sigma)$-forest property.
}\end{defin}
In what follows, we determine the chromatic number of some tree graph ${\rm KG}(G, {\cal G}, {\cal T}_t)$
provided that the graph $G$ has ${\cal G}$-forest property. Hence, it would be of interest to know when a graph satisfies 
this property. 

Assume that ${\cal G}=\{G_1,G_2,\ldots,G_k\}$ is a decomposition of a graph $G$. A cycle $C$ of $G$ is termed a {\it rainbow cycle}, if for any $1\leq i\leq k$, $|E(C)\cap E(G_i)|\leq~1$.
\begin{lem}\label{forestsubgraph}
Let $k$ be a positive integer. 
Assume that ${\cal G}=\{G_1,G_2,\ldots,G_k\}$ is a 
decomposition of a graph $G$. If for any rainbow cycle $C$ of $G$, one of the following conditions holds,  
\begin{enumerate}
\item[{\rm a)}] $\displaystyle\sum_{E(G_i)\cap E(C)\neq \varnothing}|E(G_i)|>  {2|E(C)|\over |E(C)|-1}\displaystyle
{2k-3\choose 2}+|E(C)|,$
\item[{\rm b)}] there exists an $1\leq i\leq k$ such that $E(G_i)\cap E(C)\not =\varnothing$ and that 
$|E(G_i)|-|E(G_1)|\geq  2\displaystyle {2k-3\choose 2}+1$, 
\end{enumerate}
then $G$ has ${\cal G}$-forest property. 
\end{lem}
\begin{proof}{
Let $\sigma \in S_{E(G)}$ be a permutation
such that the edges of each $G_i$ are consecutive in $\sigma$. Also, consider an alternating $2$-coloring
of $E(G)$ of the length at least $|E(G)|-|E(G_1)|+1$ with respect to the ordering $\sigma$.
Without of loss of generality, suppose that $G^R$ is a connected spanning ${\cal G}$-subgraph. 
It is readily seen that the assertion holds for $k\leq 2$. Therefore, assume $k\geq 3$. 

Consider a spanning ${\cal G}$-subgraph $F$ of $G^R$ with $k$  edges and the minimum number of cycles. Note that some vertices of 
$F$ might be isolated.  
We show that $F$ is a spanning forest with $k$ edges. 
Since $F$ has exactly $k$ edges,  for any $1\leq j\leq k$, we have $|E(F)\cap E(G_{j})|=1$.
Assume that for $j=1,2,\ldots,k$, $E(F)\cap E(G_{j})=\{e_j\}$.  
For a contradiction, suppose that $C_t$ is a rainbow cycle of length $t$ of $F$. 
Also, suppose that  for any $1\leq j\leq t$, 
$E(C_t)\cap E(G_{i_j})\not=\varnothing$ where $|E(G_{i_1})|\leq |E(G_{i_2})|\leq \cdots \leq |E(G_{i_t})|$. 
The number of neutral edges of $G$ is at most $|E(G_1)|-1\leq |E(G_{i_1})|-1$. Now we show that if condition~(a) or~(b) 
for the rainbow cycle $C_t$ holds, then the number of red edges in  
$\bigcup_{j=1}^t E(G_{i_j})$ is greater than ${2k-3\choose 2}$. 
To see this, first suppose that condition~(a) for the 
rainbow cycle $C_t$ holds. 
Suppose that there are $n_j$ neutral edges in $E(G_{i_j})$, for $j=1,2,\ldots,t$. In view of the ordering of $\sigma$, one can see that the difference between the number of red edges and the number of blue edges of $G_{i_j}$ is at most one. Consequently, 
there are at least 
$$\sum_{j=1}^t{|E(G_{i_j})|-n_j-1\over 2}\geq \sum_{j=1}^t{|E(G_{i_j})|\over 2}-{|E(G_{i_1})|\over 2}-{t-1\over 2}> {2k-3\choose 2}$$
red edges in $\bigcup_{j=1}^t E(G_{i_j})$.\\
Now suppose that condition~(b) for the 
rainbow cycle $C_t$ holds. 
In this case, there is a $j\in\{1,2,\ldots,t\}$ such that 
there are at least $|E(G_{i_j})|-|E(G_1)|+1\geq 2\displaystyle {2k-3\choose 2}+2$ red and blue edges in $E(G_{i_j})$.
Accordingly, there are at least $\displaystyle {2k-3\choose 2}+1$ red edges in $G_{i_j}$.

The graph $F$ has at most $2k-3$ non-isolated vertices, and therefore, there exists a red edge $e\in \bigcup_{j=1}^t E(G_{i_j})$ incident with some isolated vertex of $F$. Assume that $e\in E(G_{i_l})$ and consider 
$F'=(F- \{e_{i_l}\})\cup \{e\}$. One can see that $F'$ is a ${\cal G}$-subgraph of $G^R$ with $k$ edges and it has 
fewer cycles than $F$, which is impossible.
}\end{proof}
If $k\leq 2$, then $G$ has no rainbow cycle. Hence, in view of the aforementioned lemma, the following lemma holds. 
\begin{lem}\label{lastlem}
Let $k$ be a positive integer. 
Assume that ${\cal G}=\{G_1,G_2,\ldots,G_k\}$ is a 
decomposition of a graph $G$. If at least one of the following conditions holds, then $G$ has ${\cal G}$-forest property
\begin{itemize}
\item $k\leq 2$, 

\item  $k=3$ and $|E(G_3)|\geq  |E(G_1)|+7$,

\item $\displaystyle\sum_{i=1}^{3}|E(G_i)|\geq  {3}
{2k-3\choose 2}+4$,

\item there is no rainbow cycle in $\displaystyle\bigcup_{i=1}^{k-1}G_i$ and $|E(G_k)|-|E(G_1)|\geq 2\displaystyle {2k-3\choose 2}+1$.
\end{itemize}
\end{lem}

In the next lemma, we introduce a sufficient condition to extend a  ${\cal G}$-forest of a connected graph $G$ to a ${\cal G}$-tree such that the number of vertices of  ${\cal G}$-tree is sufficiently 
less than the number of vertices of $G$. 
\begin{lem}\label{treesubgraph}
Let $k$, $n$, and $r$ be positive integers, where $n\geq k+r+1$.
Assume that $G$ is a connected graph with $m$ vertices, where $n-r\leq m\leq n$. 
Also, let ${\cal G}=\{G_1,G_2,\ldots,G_k\}$ be a family of 
pairwise  edge-disjoint 
nonempty subgraphs of $G$. If $G$ has  a  ${\cal G}$-forest and $|E(G)|> {k(3k-2)(n-2)\over 2}+ (r+k-1)(n-1)$, then $G$ has a  ${\cal G}$-tree with $n-r$ vertices. 
\end{lem}
\begin{proof}{
Consider a spanning  ${\cal G}$-forest $F$. Note that some vertices might be isolated. 
Since $G$ is connected, we can add some edges to $F$ to obtain a spanning ${\cal G}$-tree.  
Now it is enough to show that $G$ contains a ${\cal G}$-tree $T$ with at most $n-r$ vertices.  
Note that if we prove this, then we can extend $T$ (if it is necessary) to a tree with exactly $n-r$ vertices, which implies the assertion. Consider all ${\cal G}$-trees with the minimum number of vertices. Among these ${\cal G}$-trees, 
let $T$ be a ${\cal G}$-tree with the maximum number of pendant vertices, i.e., the vertices of degree one. 
If $|V(T)|\leq n-r$, then there is nothing to prove.
Therefore, we  can assume that $|V(T)|>n-r$. Also, if $T$ has at least $k+1$ pendant vertices, 
then there is a pendant vertex $v$ such that $T - v$ is still a ${\cal G}$-tree 
with fewer vertices, which contradicts the minimality of $T$.
Hence, $T$ has at most $k$ pendant vertices. 
It is known that for any tree $T$, the number of pendant vertices is  
$$2+\sum_{\{v\in V(T):\ {\rm deg}_T(v)\geq 3\}}({\rm deg}_T(v)-2).$$ 
This implies that $T$ has at most $k-2$ vertices of degree more than $2$ and the maximum degree of $T$ is at most $k$.
Assume that $\{v_1,v_2,\ldots,v_t\}$ are the pendant vertices of $T$ where $2\leq t\leq k$. 
Set $W=V(T)\setminus \{v_1,v_2,\ldots,v_t\}$.
Since $T$ has at most $k$ pendent vertices, 
for any vertex $u\in W$ and any positive integer $i$, we have $|\{v: d_T(u,v)=i\}|\leq k$, where $d_T(u,v)$ denotes the distance 
between $u$ and $v$ in $T$. Therefore,
there are at most $k(3k-2)$  vertices $v\in W$ such that $1\leq d_T(u,v)\leq 3k-2$.
In view of the assumption and since $|W|\geq n-r-t+1\geq n-r-k+1$, one can check that the graph $G[W]$ has
at least ${k(3k-2)(n-2)+1\over 2}$ edges. 
Therefore, since $|W| \leq n-t \leq n-2$, there is an edge $e\in G[W]$
such that the unique cycle $C$ in $T\cup \{e\}$ has the length at least $3k$.
Since $T$ has at most $k-2$ vertices whose degrees are at least $3$, there are at least
$k+1$ edges in $C\setminus\{e\}$ such 
that the degree of any vertex of these edges in the graph $T\cup\{e\}$ is $2$, and also, these edges are~not incident with the edge $e$. 
Now, if there is an edge $f\in E(G)\setminus\bigcup_{i=1}^kE(G_i)$, then we set 
$T'=(T\cup\{e\})-\{f\}$. 
Otherwise, there are two edges $f$ and $f'$ among the aforementioned edges
such that $f,f'\in E(G_l)$ for some $l\in\{1,2,\ldots,k\}$. Now set $T'=(T\cup\{e\})-\{f\}$.
One can see that $T'$ is a ${\cal G}$-tree with the same vertex set as $T$. Also, the number of pendant vertices of $T'$
is more than that of $T$, which contradicts our assumption.
}\end{proof}
Assume that $G$ is a graph and $\sigma$ is an ordering of $E(G)$, i.e. $\sigma\in S_{E(G)}$. 
For a vertex $v\in V(G)$, a path $uvw$ is celled a {\it $(\sigma,v)$-consecutive path} 
if the edges $uv$ and $vw$ are consecutive in the ordering $\sigma$. 
In any alternating $2$-coloring of the edges of $G$ with respect to the ordering $\sigma$, we do~not assign the same color 
to the edges of any $(\sigma,v)$-consecutive path. 
This implies the next lemma.
\begin{lem}\label{consecutive}
Let $G$ be a graph, $v$ be a vertex of $G$, and $\sigma$ be an ordering of the edges of $G$. 
Consider an alternating $2$-coloring of the edges of $G$ with respect to the ordering $\sigma$ and 
let $E_v^R$ be a subset of $E(G)$ such that any red edge incident with $v$ is a member of $E_v^R$.  
If there are $t_1$ edge-disjoint $(\sigma,v)$-consecutive paths such that the number of 
these $(\sigma,v)$-consecutive paths  which have nonempty 
intersection with $E_v^R$ is at most $t_2$, 
then there are at least $t_1-t_2$ neutral edges incident with $v$ in $E(G)\setminus E_v^R$. 
\end{lem}
Suppose that $G$ is a subgraph of an Eulerian graph $H$. An ordering $\sigma$ of $E(G)$ is induced by an Eulerian tour of $H$, if 
the edges of $G$ are ordered corresponding to their ordering in the Eulerian tour of $H$, i.e.,
if we traverse the edge $e$ before the edge $e'$ in the Eulerian tour of $H$, then in the ordering $\sigma$ we have $e < e'$. 

\begin{lem}\label{orderinglemma}
Let $H$ be an Eulerian graph and $G$ be a subgraph of $H$. 
If $\sigma$ is an ordering of the edge set of $G$ induced by an Eulerian 
tour of $H$, then in any alternation $2$-coloring of $E(G)$ with respect to the ordering $\sigma$, the number of red {\rm(}resp. blue{\rm)} 
edges incident with a vertex $v$ of $G$ is at most ${{\rm deg}_H(v)+2\over 2}$. 
\end{lem}
\begin{proof}{
We show that for each vertex $v\in V(G)$, we have ${\rm deg}_{G^R}(v)\leq {{\rm deg}_H(v)+2\over 2}$, and similarly, the assertion holds for the blue spanning subgraph $G^B$. 
In view of the definition of alternating coloring, we do~not assign the same color to the edges of any 
$(\sigma,v)$-consecutive path. One can check that in the ordering $\sigma$, the number of 
$(\sigma,v)$-consecutive paths in $G$ is at least ${2{\rm deg}_G(v)-{\rm deg}_H(v)-2\over 2}$. 
Note that if the vertex $v$ is~not the beginning vertex of the Eulerian tour, then this number is at least 
${2{\rm deg}_G(v)-{\rm deg}_H(v)\over 2}$.
This implies that there are at least  ${2{\rm deg}_G(v)-{\rm deg}_H(v)-2\over 2}$ edges incident with $v$ 
which are~not red. Hence, 
the number of red edges incident with the vertex $v$ is at most ${{\rm deg}_H(v)+2\over 2}$. 
}\end{proof}
Now we are in a position to prove the main lemma.
\begin{lem}\label{main}
Let $k, n$, and $r$ be positive integers and $\delta $ be a real number, where 
${5\over 6}< \delta < 1$, $k=o(\sqrt{n})$, 
and $r=o(n)$. There is a constant integer
$N=N(k,\delta,r)$ such that for any $n\geq N$ we have the following. 
Assume that $G$ is a graph with $n\geq N$ vertices and ${\cal G}=\{G_1,G_2,\ldots,G_k\}$ is a 
decomposition of $G$, 
where $|E(G_k)|\geq |E(G_{k-1})|\geq \cdots\geq |E(G_1)|\geq 1$. If 
the following conditions hold
\begin{itemize}
\item $G$ has ${\cal G}$-forest property, 
\item $G_k$ is a spanning subgraph of $G$ with $\delta(G_k)\geq \delta n$, 
\end{itemize}
then
$$\chi({\rm KG}(G, {\cal G}, {\cal T}_{ n-r}))={\rm cut}_{r+1}(G, {\cal G}).$$
\end{lem}
\begin{proof}{
Assume that $n\geq N$, where $N=N(k,\delta,r)$ is a constant and will be determined during the proof. 
Note that if there is some $i\in\{1,2,\ldots,k-1\}$ such that $|E(G_i)|=1$, then we have
$\chi({\rm KG}(G, {\cal G}, {\cal T}_{n-r}))\leq 1$ and
${\rm ex}(G, {\cal G}, {\cal T}_{n-r})\geq |E(G)|-1$. Now it is easy to see that 
$\chi({\rm KG}(G, {\cal G}, {\cal T}_{ n-r}))={\rm cut}_{r+1}(G, {\cal G})$. 
Therefore, we can assume that for any $i\in\{1,2,\ldots,k-1\}$, we have $|E(G_i)|\geq 2$.
In view of Lemma~\ref{ex}, it is enough to show that for large enough $N$, 
${\rm ex}_{alt}(G,{\cal G}, {\cal T}_{n-r})\leq |E(G)|-{\rm cut}_{r+1}(G,{\cal G})={\rm ex}(G,{\cal G},{\cal T}_{n-r})$, which 
is equivalent to the following claim.\\ 

{\bf Main Claim:} There is an ordering $\sigma\in S_{E(G)}$ such that for any alternating $2$-coloring of $E(G)$ of length ${\rm ex}(G,{\cal G},{\cal T}_{n-r})+1$ with respect to the ordering $\sigma$,  
$G^R$ or $G^B$ contains a ${\cal G}$-tree with $n-r$ vertices. 

The proof will be divided into $5$ steps as follows. In the first step, we introduce the ordering $\sigma$. In the second step, 
we consider an alternating $2$-coloring for the edge set of $G$ of length ${\rm ex}(G,{\cal G},{\cal T}_{n-r}, \sigma)+1$ with respect to the ordering $\sigma$. 
Furthermore, we present some upper bounds for  
the $r^{th}$ cut number of $G$ and the maximum degrees of $G^R$ and $G^B$. 
In the third step, we show that if $G^R$ or $G^B$ is~not a ${\cal G}$-subgraph, then the main claim follows. 
In the fourth step, we claim that $G^R$ or $G^B$ has a 
large connected component and we prove this claim by two cases. Finally, in the last step, we show  
the largest connected component of $G^R$ or $G^B$ contains a  ${\cal G}$-tree with $n-r$ vertices, and consequently, 
the main claim holds.\\

{\bf Step I:} An ordering for the edge set of $G$\\ 
For any $i\in\{1,2,\ldots,k-1\}$, if $G_i$ has some odd degree vertices, then add a new vertex
$z_i$ to $G_i$ and join it to all odd degree vertices in $G_i$. Otherwise, consider $G_i$ itself. Suppose that $G_i^1,\ldots,G_i^{t_i}$ are the connected component of 
the resulting graph. For any $1\leq j\leq t_i$, consider an Eulerian tour for $G_i^j$ and define $\sigma_i^j$
to be an ordering of $E(G_i^j)\cap E(G_i)$
such that these edges are ordered corresponding to their ordering in the aforementioned 
Eulerian tour of $G_i^j$, i.e., if the edge $e_i$ is traversed before the edge $e_j$, then $e_i<e_j$. Define the 
ordering $\sigma_i$ for the edge set of $G_i$ as follows 
$$\sigma_i=\sigma_i^1||\sigma_i^2||\cdots||\sigma_i^{t_i}.$$ 

Let $V(G_k)=V(G)=\{v_1,v_2,\ldots,v_n\}$.  
Here, we want to present an ordering for the edges of $G_k$.
Set $l=\lceil {26r+26\over \delta}+{3k\over 2}\rceil=o(n)$ and  
assume that $n=4a+b$, where $a$ and $b$ are positive integers and $1 \leq b\leq 4$. 
Join $4-b$ new vertices to the graph $G_k$ to obtain the graph $G_k'=G_k\vee K_{4-b}$. 
Let $F$ be a graph with exactly $a+1$ connected components such that each connected component of 
$F$ is isomorphic to the complete graph $K_4$. 
By Theorem \ref{Hajnal}, the $F$-packing number of $G_k'$ is at least ${1\over 36}n$, and consequently, 
if $N$ is large enough, then it is more than 
$4l+2$. 
Assume that the set $\{F_1,F_2,\ldots,F_{4l+3}\}$ forms an $F$-packing of $G_k'$.  
For any $i=1,2,\ldots,4l+3$ 
and for any connected component of $F_i$, which is a subgraph of $G$, choose a  four cycle $C_4$  of this component. 
One can see that every vertex of $G_k$ appears in at least $4l$ and at most $4l+3$ of these $C_4$'s. 
Clearly, these $C_4$'s form a monogamous $C_4$-packing of $G_k$. 

Call every $C_4$ of this packing a {\it $4$-block}.
Construct a bipartite graph with the vertex set $W\cup W'$ such that $W$ consists of
$l$ copies of each vertex in $V(G)$
and $W'$ consists of all $4$-blocks. 
Join a vertex of $v\in W$ to a vertex $u\in W'$, if the corresponding vertex of $v$ in 
$V(G)$ is contained in the corresponding $4$-block of $u$. 
One can check that the degree of every vertex in the part $W$ is   at least $4l$ and also the degree of every vertex in the part $W'$ is exactly $4l$. 
In view of Hall's Theorem, one can see that this bipartite graph has a matching, which saturates all vertices of $W$. 
For any 
vertex $v_i\in V(G)$, set $A_i$ to be the set of all $4$-blocks assigned to all copies of $v_i$ through the aforementioned matching.
Note that $|A_i|=l$ and since $4$-blocks form a monogamous $C_4$-packing, 
the intersection of any two $4$-blocks in $A_i$ is $v_i$.  
For each $i\in[n]$, set $H_i$ to be a subgraph of $G_k$ consisting of all $4$-blocks in $A_i$.
One can see that 
$H_1,\ H_2,\ldots,\ H_n$ are pairwise edge-disjoint subgraphs of $G_k$ 
and for each $i\in [n]$, $v_i\in V(H_i)$. Also, every vertex of $V(H_i)\setminus \{v_i\}$ has degree $2$ in $H_i$. 

In view of the aforementioned $F$-packing, one can check that the degree of every vertex of $G_k \setminus\displaystyle\cup_{i=1}^n H_i$ is 
at least $\delta n-8l-6$. Now by a well-known result of Dirac, one can check that for large enough $N$, the graph 
$G_k \setminus\displaystyle\cup_{i=1}^n H_i$ contains four edge-disjoint Hamiltonian cycles $C_n, C_n', C_n''$ and $C_n'''$. 
Without loss of generality, we can assume that $E(C_n)=\{v_iv_{i+1}: 1\leq i\leq n-1\}\cup\{v_1v_n\}$.
If $G_k$ has no odd degree vertices, then set $M=\varnothing$. 
Otherwise, consider a family $M$ of vertex-disjoint paths in $C_n'$ such that every odd degree vertex of $G_k$ is an end point of these paths and also the end points of these paths have odd degree in $G_k$. 

By a result of Huu Hoi~(see page 8 in \cite{HHHoi}), there are at least 
\begin{equation}\label{numbertriangle}
{\epsilon\over 3\nu(\nu-2)}(4\epsilon-\nu^2)-O(\nu)
\end{equation}
edge-disjoint triangles in a graph with $\nu$ vertices and $\epsilon$ edges. Consider the spanning subgraph $L$ of $G_k$ whose edge set is  
$E(M)\cup E(C_n) \cup E(C_n'')\cup E(C_n''')\cup(\displaystyle\bigcup_{i=1}^nE(H_i))$.  
Define $y=|E(L)|$. One can see that $|E(L)|\leq (4l+4)n$. In view of~(\ref{numbertriangle}), consider at least 
\begin{equation}\label{numbertriangle2}
{|E(G_k)|-y\over 3n(n-2)}(4|E(G_k)|-4y-n^2)-O(n)=
{|E(G_k)|\over 3n(n-2)}(4|E(G_k)|-n^2)-o(n^2)\end{equation}
edge-disjoint triangles in $G_k\setminus L$. 


Call each of these triangles  a $3$-block.
Let $G_k''$ be the graph obtained from $G_k$ by removing all edges in $L$ and
the aforementioned  $3$-blocks. Every vertex of $G_k''$ has an even degree, and therefore, every connected component of $G_k''$ is an Eulerian graph.
Assume that $Q_1, Q_2, \ldots, Q_s$ are connected components of $G_k''$. 

We construct an Eulerian tour for $G_k\setminus (M\cup C_n''\cup C_n''')$ 
by the following algorithm and using it we present
an ordering $\sigma_k$ for the edge set of $G_k$.
At the $i^{th}$ stage, where $1\leq i \leq n$, do the following steps
\begin{itemize}
\item Traverse an Eulerian tour of $H_i$ started at $v_i$.
\item Traverse every $3$-block containing $v_i$ which is still untraversed.
\item If there is a $j\in[s]$ such that $v_i\in Q_j$ and the edge set of $Q_j$ is still untraversed, then consider 
an Eulerian tour for $Q_j$ starting at $v_i$ and traverse it.
\item Traverse the the edge $v_iv_{i+1}$ (the indices are taken modulo $n$).
\item If $i<n$, then start the $(i+1)^{th}$ stage. 
\end{itemize}

Assume that $E(C_n'')=\{f_1,f_2,\ldots,f_n\}$ and $E(C_n''')=\{f'_1,f'_2,\ldots,f'_n\}$ such that
for each $i\in\{1,2,\ldots,n\}$, two edges $f_i$ and $f_{i+1}$ (resp. $f'_i$ and $f'_{i+1}$)
are incident (the indices are taken modulo $n$).  
Construct an ordering $\pi$ for the edge set of the graph $G_k\setminus (M\cup C_n''\cup C_n''')$
such that the edges of $G_k\setminus (M\cup C_n''\cup C_n''')$ are ordered 
corresponding to their ordering in the aforementioned Eulerian tour, i.e.,
if we traverse the edge $e$ before the edge 
$e'$ in the Eulerian tour, then in the ordering $\pi$, we have $e<e'$.
Construct the ordering $\sigma_k$ from $\pi$ by putting the edges of $M$ at the end of the ordering $\pi$  such that any two incident edges in $M$ are consecutive.
Set 
$$\sigma=\sigma_1||\sigma_2||\cdots||\sigma_k||f_1<f'_1<f_2<f'_2<\ldots<f_n<f'_n.$$

{\bf Step II:} The $r^{th}$ cut number of $G$ and the maximum degrees of $G^R$ and $G^B$\\
Consider an alternating $2$-coloring of edges of $G$ with respect to the ordering $\sigma$ of length 
${\rm ex}(G,{\cal G},{\cal T}_{n-r})+1$. 
In view of the definition of ${\rm ex}(G, {\cal G}, {\cal F})$, we have 
$${\rm ex}(G, {\cal G}, {\cal T}_{n-r})\geq \sum_{j=1}^k|E(G_j)|-|E(G_1)|=|E(G)|-|E(G_1)|.$$
Also, in the aforementioned alternating $2$-coloring of $E(G)$, the number of neutral edges is equal to 
$${\rm cut}_{r+1}(G)-1=|E(G)|-{\rm ex}(G, {\cal G}, {\cal T}_{n-r})-1\leq |E(G_1)|-1.$$
Also, one can check that 
$${\rm cut}_{r+1}(G,{\cal G})\leq \min\{|E_G(S,V(G)\setminus S)|:\ S\subseteq V(G),\ \ |S|\geq r+1,\ |V(G)\setminus S|\geq r+1 \}.$$
Therefore, 
$${\rm cut}_{r+1}(G,{\cal G})\leq \min(\displaystyle\{|E(G_1)|,  (r+1)(n-r-1)\})=o(n^2).$$

In view of the ordering $\sigma$ and Lemma~\ref{orderinglemma}, one can check that 
for any vertex $v\in V(G)$, since $G_k\setminus (M\cup C_n''\cup C_n''')$ is an Eulerian graph,
the number of red (resp. blue) edges incident with $v$ in $G_k\setminus (M\cup C_n''\cup C_n''')$ is at most 
${({\rm deg}_{G_k}(v)-{\rm deg}_{M}(v)-4)+2\over 2}$.
According to the ordering of the edges of $M\cup C_n''\cup C_n'''$ in $\sigma$,
all edges of $C_n''\cup C_n'''$  incident with $v$  and at most one edge of $M$ incident with $v$  (if there is such an edge) might be red (resp. blue). 
Consequently, since ${\rm deg}_{M}(v)\in\{0,1,2\}$, 
\begin{equation}\label{degGk}
\max\{{\rm deg}_{G_k^R}(v), {\rm deg}_{G_k^B}(v)\}\leq{{\rm deg}_{G_k}(v)+7\over 2}. 
\end{equation} 

Also, in view of Lemma~\ref{orderinglemma}, the number of red edges incident with $v$ in $G_i$, for $1\leq i\leq k-1$,  is at most 
${{\rm deg}_{G_i}(v)+3\over 2}$. Hence, the degree of $v$ in $G^R$ is at most $ {{\rm deg}_{G_k}(v)+7\over 2}+\displaystyle\sum_{i=1}^{k-1}{{\rm deg}_{G_i}(v)+3\over 2}$. Similarly, one can obtain the same bound for the blue subgraph $G^B$, and therefore, 
\begin{equation}\label{degree}
\max\{{\rm deg}_{G^R}(v), {\rm deg}_{G^B}(v)\}\leq  {{\rm deg}_G(v)+3k+4\over 2}.
\end{equation}\\

{\bf Step III:} Both of $G^R$ and $G^B$ are ${\cal G}$-subgraphs.\\
In this step, we show that if $G^R$ or $G^B$ is~not a ${\cal G}$-subgraph, then the main claim follows. 
First, assume that there exists an $i_0\in\{1,2,\ldots,k\}$ such that $G_{i_0}$ has no blue 
edge, i.e., $G^B$ is~not a ${\cal G}$-subgraph. 
Since there are at most $|E(G_1)|-1$ neutral edges, hence, in view of the ordering of the edge set of $G_{i_0}$, there are   
exactly $|E(G_{i_0})|-1$ neutral edges in $G_{i_0}$ and that $|E(G_{i_0})|=|E(G_{1})|$. In particular, $G_{i_0}$ contains exactly 
one red edge. This also 
implies that there is no neutral edges in $E(G)\setminus E(G_{i_0})$. 
Therefore,  either $C_n''$ or $C_n'''$ has no blue edge.
Without loss of generality, assume that all edges of $C_n''$ are red.  Also, 
note that for any 
$i\in\{1,2,\ldots,k\}\setminus\{i_0\}$, $G_i$ has at least two edges and in view of the ordering of its edges, $G_i$ has at least a red edge. 
For any $i\in\{1,2,\ldots,k\}$, let $g_i$ be a red edge in $G_i$. 
Now consider the red subgraph $C_n''\cup\{g_1,g_2,\ldots,g_{k}\}$.
Clearly, this subgraph is a connected spanning subgraph of $G^R$, and 
consequently, $G^R$ is a connected spanning graph, which intersects all $G_i$'s for $i=1,2,\ldots,k$. 
Since $G$ has ${\cal G}$-forest  property, one can conclude that $G^R$ contains a  ${\cal G}$-forest. 
Also, since the number of neutral edges is at most ${\rm cut}_{r+1}(G,{\cal G})-1$, 
one can see that 
\begin{equation}\label{coloredges}
\min\{|E(G_k^R)|, |E(G_k^B)|\}\geq {|E(G_k)|-{\rm cut}_{r+1}(G,{\cal G})\over 2}> {5\over 24}n^2-{1\over 2}(r+1)(n-r-1).
\end{equation}
Also, for large enough $N$, 
$${5\over 24}n^2-{1\over 2}(r+1)(n-r-1)> {k(3k-2)(n-2)\over 2}+ (r+k-1)(n-1).$$
Hence, in view of Lemma~\ref{treesubgraph}, $G^R$ contains a ${\cal G}$-tree, which is a member of ${\cal T}_{n-r}$. 
This implies the main claim. Similarly, if there exists an  $i_0\in\{1,2,\ldots,k\}$ such that $G_{i_0}$ has no red 
edge, then the main claim holds. 
Thus we can suppose that for any $i\in\{1,2,\ldots,k\}$, both of colors appear in $E(G_i)$, i.e., 
both of $G^R$ and $G^B$ are ${\cal G}$-subgraphs.\\

{\bf Step IV (Claim):} $G^R$ or $G^B$ has a connected component with at least $n-r$ vertices.\\
One the contrary, suppose that every connected component of $G^R$ and also $G^B$ has at most $n-r-1$ vertices. 
Assume that $\Omega_R$ (resp. $\Omega_B$) is the largest connected component of $G^R$ (resp. $G^B$). 
In view of the assumption, we have $\max\{|V(\Omega_R)|, |V(\Omega_B)|\}\leq n-r-1$. 
Therefore, in view of (\ref{coloredges}), 
the number of red and blue edges in $G_k$ is $O(n^2)$, i.e., $|E(G^R)|=O(n^2)$ and $|E(G^B)|=O(n^2)$. 
Since $r=o(\sqrt n)$, for large enough $N$, we have $\min\{|V(\Omega_R)|, |V(\Omega_B)|\}\geq r+1$. 
Let $H_R$ be the bipartite subgraph of $G$ with the vertex set $(U_R,V_R)$,
where $\{U_R, V_R\}=\{V(\Omega_R), V(G) \setminus V(\Omega_R)\}$ and
$|U_R|\leq|V_R|$. 
A vertex of $U_R$ is a neighbor of a vertex 
of $V_R$, if they are neighbors in $G$. Similarly, the bipartite subgraph 
$H_B$ of $G$ has $(U_B,V_B)$ as vertex set, where  $\{U_B, V_B\}=\{V(\Omega_B), V(G) \setminus V(\Omega_B)\}$ and
$|U_B|\leq|V_B|$. 
Set $\alpha=|U_R|$ and $\beta=|U_B|$ and note that $r+1\leq \alpha, \beta\leq {n\over 2}$. Consider the bipartite graph $H_R$. 
Note that there is no red edge in $E(H_R)=E_G(U_R,V_R)$. 
Therefore, in view of (\ref{degree}),    

$$
\begin{array}{rcl}
|{\rm NEU}(G)| & \geq &
\displaystyle\sum_{u\in U_R} ({\rm deg}_{G}(u)-{\rm deg}_{G^R}(u)-{\rm deg}_{G^B}(u))\\ 
& & \\
& \geq & \displaystyle\sum_{u\in U_R} ({\rm deg}_{G}(u)-(\alpha-1)-{{\rm deg}_{G}(u)+3k+4\over 2})\\ 
& & \\
& \geq &  \displaystyle\sum_{u\in U_R} ({{\rm deg}_{G}(u)\over 2}-\alpha-{3k+2\over 2})\\
& & \\
&\geq & {1\over 2} \delta n\alpha-\alpha^2-{3k+2\over 2}\alpha.
\end{array}
$$

Since there are ${\rm cut}_{r+1}(G)-1$ neutral edges in $G$ and that ${\rm cut}_{r+1}(G)-1\leq (n-r-1)(r+1)-1$, we should have  
$(n-r-1)(r+1)-1\geq {1\over 2} \delta n\alpha-\alpha^2-{3k+2\over 2}\alpha$. 
This implies that either $r+1\leq \alpha\leq {3(r+1)\over \delta}$ or ${\delta\over 2}n-o(n)\leq \alpha\leq {n\over 2}$ provided that $N$ is large enough. Similarly, we have either $r+1\leq \beta\leq {3(r+1)\over \delta}$ or ${\delta\over 2}n-o(n)\leq \beta\leq {n\over 2}$ provided that $N$ is large enough.\\ 

{\bf Case I:} ${\delta\over 2}n-o(n)\leq \min\{|V(\Omega_t)|, |V(G) \setminus V(\Omega_t)|\} \leq {n\over 2};\ t\in \{R,B\}$\\
First,  assume that 
${\delta\over 2}n-o(n)\leq \alpha\leq {n\over 2}$ and consider the bipartite graph $H_R$. Note that  
$$|E(H_R)\cap E(G_k)|\geq \alpha({\delta n}-\alpha+1)\geq  \alpha({\delta n}-\alpha)\geq ({2\delta-1\over 4})n^2,$$ 
provided that $N$ is large enough. 
Since, there is no red edges between two parts  $U_R$ and $V_R$, and moreover,  
there are at most $(n-r-1)(r+1)-1=o(n^2)$ neutral edges between them, one can conclude that there are at least $({2\delta-1\over 4})n^2-o(n^2)$ blue edges in $E(H_R)\cap E(G^B_k)$. 

Call an edge of $G_k$ an {\it unusual} edge, if it is either a neutral edge in $G_k$ or a blue edge of $E(G_k)\setminus E(H_R)$. 
Also, we call two $3$-blocks {\it consecutive}, if the edge set of these blocks appear 
consecutively in the ordering $\sigma$. In other words, there are no other edges among them in the ordering $\sigma$. 
One can check that there exists at least an unusual edge in any two consecutive $3$-blocks of $G_k$. 
In view of (\ref{numbertriangle2}), the ordering $\sigma$, and the aforementioned discussion, one can check that the number of unusual edges is at least 
$${1\over 2}\left({|E(G_k)|\over 3n(n-2)}(4|E(G_k)|-n^2)-o(n^2)\right)-O(n)={|E(G_k)|\over 6n(n-2)}(4|E(G_k)|-n^2)-o(n^2).$$ 

Since the number of neutral edges is $o(n^2)$,  one can conclude that the number of blue edges in $E(G^B_k)\setminus E(H_R)$ is at least $|{E(G_k)|\over 6n(n-2)}(4|E(G_k)|-n^2)-o(n^2)$. Accordingly, the number of blue edges in $G_k$ is 
at least 

$$\begin{array}{lll}
|E(G^B_k)\setminus E(H_R)| & + & |E(H_R)\cap E(G_k^B)| \geq \\
&&\\
|{E(G_k)|\over 6n(n-2)}(4|E(G_k)|-n^2)-o(n^2) & + & {2\delta -1 \over 4}n^2.
\end{array}$$

Set $\bar{\delta}n$ to be the average degree of the graph $G_k$ for which we have $|E(G_k)|={\bar{\delta}n^2\over 2}$. 
Note that $\bar{\delta}\geq \delta>{5\over 6}$. One can check that  there is a $\mu=\mu(\delta)>0$ such that 
$${|E(G_k)|\over 6n(n-2)}(4|E(G_k)|-n^2)+{2\delta -1 \over 4}n^2 > (1+\mu){|E(G_k)|\over 2},$$ 
or equivalently,  
$${\bar{\delta}n^2\over 12n(n-2)}(2\bar{\delta}n^2-n^2)+{2\delta -1 \over 4}n^2 > (1+\mu){\bar{\delta}n^2\over 4}.$$
To see this, one can check if $\mu$ is sufficiently small, then the following inequality holds for any real number $x$ 
$${1\over 6}x^2-({1\over 3}+ {\mu\over 4})x+{2\delta -1 \over 4}>0.$$

Therefore, for large enough $N$, 
$${\bar{\delta}n^2\over 12n(n-2)}(2\bar{\delta}n^2-n^2)+{2\delta -1 \over 4}n^2-o(n^2) > {\bar{\delta}n^2\over 4}+1={|E(G_k)|\over 2}+1.$$
This contradicts this fact that the number of blue edges in $G_k$ is at most ${|E(G_k)|\over 2}+1$. 
Similarly, if ${\delta\over 2}n-o(n)\leq \beta\leq {n\over 2}$, we get a contradiction.\\

{\bf Case II:} $r+1 \leq \min\{|V(\Omega_t)|, |V(G) \setminus V(\Omega_t)|\} \leq {3(r+1)\over \delta};\ t\in \{R,B\}$\\
In view of the aforementioned discussion, we can assume $r+1\leq \alpha, \beta \leq {3(r+1)\over \delta}$.
Note that for any vertex $x\in U_R\cap U_B$, all edges in $G$ and between $x$ and the vertices of 
$V(G)\setminus (U_R\cup U_B)$ are neutral, i.e., each edge in $E_{G}(U_R\cap U_B, V(G)\setminus (U_R\cup U_B))$ is neutral.
One can readily check that if $|U_R\cap U_B|\geq{2(r+1)\over\delta}$ and $N$ is sufficiently large,  
then the number neutral edges in $G$ is at least 
$$|U_R\cap U_B|(\delta n-|U_R\cup U_B|)\geq {2(r+1)\over \delta}(\delta n-{6(r+1)\over\delta})>(r+1)(n-r-1)$$ 
provided that $N$ is large enough, 
which is impossible. Therefore, we can suppose $|U_R\cap U_B|\leq {2(r+1)\over \delta}$. 
Now consider a vertex $v_i\in U_R\setminus U_B$.
Since all of $4$-blocks are chosen from a monogamous packing, one can check that the number of $4$-blocks in $H_i$ having some vertices of $U_R\cup U_B\setminus \{v_i\}$ is at most $|U_R\cup U_B|-1\leq {6(r+1)\over \delta}-1$. 

One can check that any $4$-block in $H_i$ which has no red edge incident with $v_i$ contains at most one red edge. 
Suppose that there are $\gamma$ consecutive $4$-blocks of $H_i$ (with respect to the ordering $\sigma$) such that 
each of them has no vertex in $U_R\cup U_B\setminus \{v_i\}$. Clearly, these blocks have no red edges 
incident with $v_i$. Consequently, they have at most $\gamma$ red edges. 
Since the edges of these blocks are consecutive in $\sigma$, they contain at most $\gamma+1$ blue edges. This implies that there are at least $4\gamma-(2\gamma+1)=2\gamma-1$ neutral edges among the edges of these $4$-blocks. Since the number of $4$-blocks in $H_i$ containing some vertex of $(U_R\cup U_B)\setminus \{v_i\}$ is at most $|U_R\cup U_B|-1$, we have 

\begin{equation}\label{neuedgezi}
|{\rm NEU}(H_i-(U_R\cup U_B\setminus\{v_i\})|\geq 2l-3|U_R\cup U_B|+3. 
\end{equation}

Also, for any $1\leq j \leq k-1$, in view of ordering of the edge set of 
$G_j$, we can choose at least 
${{\rm deg}_{G_j}(v_i)-3\over 2}$, edge-disjoint $(\sigma,v_i)$-consecutive paths in $G_j$. 
Moreover, in view of ordering of the edge set of 
$G_k$, we can choose at least ${{\rm deg}_{G_k}(v_i)-7\over 2}$, edge-disjoint $(\sigma,v_i)$-consecutive paths in 
$G_k\setminus E(C_n''\cup C_n''')$. Choose  
$\displaystyle\sum_{j=1}^{k-1}{{\rm deg}_{G_j}(v_i)-3\over 2}+{{\rm deg}_{G_k}(v_i)-7\over 2}$ edge-disjoint $(\sigma,v_i)$-consecutive paths in $G$. 
The number of these $(\sigma,v_i)$-consecutive paths which 
have a nonempty intersection with $E(H_i)$ (resp. $E(G[U_R\cup U_B])$) is at most $l+1$ (resp. $|U_R\cup U_B|-1$). 
In view of Lemma~\ref{consecutive} (set $E_{v_i}^R=G[U_R\cup U_B]\cup H_i$), one can see that the number of neutral edges incident with $v_i$ which are~not in $E(H_i)$ or $E(G[U_R\cup U_B])$, i.e., $|{\rm NEU}(v_i, G\setminus (H_i\cup G[U_R\cup U_B]))|$, is at least 
$$\begin{array}{rll}
 &  & \displaystyle\sum_{j=1}^{k-1}{{\rm deg}_{G_j}(v_i)-3\over 2}+{{\rm deg}_{G_k}(v_i)-7\over 2}-(l+1)-(|U_R\cup U_B|-1)\\
&& \\
 & = & {1\over 2}{\rm deg}_G(v_i)-{3k\over 2}-l-|U_R\cup U_B|-2
\end{array}$$

Similarly, with the same argument, we can obtain the same assertions for any vertex $v_j\in U_B\setminus U_R$. 

Hence, in view of (\ref{neuedgezi}) and since $l=\lceil {26r+26\over \delta}+{3k\over 2}\rceil$, for any vertex $v_i\in (U_R\setminus U_B)\cup(U_B\setminus U_R)=U_R\cup U_B\setminus U_R\cap U_B$, we have 

$$\begin{array}{rll}
& &  |{\rm NEU}(v_i, G\setminus (H_i\cup G[U_R\cup U_B]))|+|{\rm NEU}(H_i-(U_R\cup U_B\setminus\{v_i\}))| \\
&&\\
 & \geq  &{1\over 2}{\rm deg}_G(v_i)-{3k\over 2}-l-|U_R\cup U_B|-2+ 2l-3|U_R\cup U_B|+3\\
&&\\
 & = & {{\rm deg}(v_i)\over 2}+l-4|U_R\cup U_B|-{3k\over 2}+1\\
 &&\\
 & > &  {{\rm deg}(v_i)\over 2}+{2(r+1)\over \delta}.
\end{array}
$$

Assume that $u, u'\in U_R\cap U_B$ and $v_i, v_j \in (U_R\cup U_B) \setminus (U_R\cap U_B)$, where  $u\neq u'$ and $v_i\neq v_j$. 
Note that 
$${\rm NEU}(u, G- (U_R\cup U_B\setminus\{u\}))\cap {\rm NEU}(u', G- (U_R\cup U_B\setminus\{u'\}))=\varnothing,$$
$${\rm NEU}(u, G- (U_R\cup U_B\setminus\{u\}))\cap {\rm NEU}(v_i, G\setminus (H_i\cup G[U_R\cup U_B]))=\varnothing,$$
$${\rm NEU}(u, G- (U_R\cup U_B\setminus\{u\}))\cap{\rm NEU}(H_i-(U_R\cup U_B\setminus\{v_i\}))=\varnothing,$$
$${\rm NEU}(v_i, G\setminus (H_i\cup G[U_R\cup U_B])) \cap {\rm NEU}(v_j, G\setminus (H_j\cup G[U_R\cup U_B]))=\varnothing,$$ 
$${\rm NEU}(H_i-(U_R\cup U_B\setminus\{v_i\})) \cap{\rm NEU}(v_j, G\setminus (H_j\cup G[U_R\cup U_B]))=\varnothing,$$ 
and 
$${\rm NEU}(H_i-(U_R\cup U_B\setminus\{v_i\})) \cap {\rm NEU}(H_j-(U_R\cup U_B\setminus\{v_j\}))=\varnothing.$$
Now by counting the number of neutral edges of $G$ incident with some vertices in $U_R\cup U_B$, we get a contradiction. 
First, note that all edges between $U_R\cap U_B$ and the vertices of $V(G)\setminus (U_R\cup U_B)$ are neutral. Also, 
$$\begin{array}{lll}
|{\rm NEU}(G)| & \geq& \displaystyle\sum_{u\in U_R\cap U_B} |{\rm NEU}(u, G- (U_R\cup U_B\setminus\{u\}))| \\ 
   &&\\
   &+& \displaystyle\sum_{v_i\in (U_R\cup U_B)\setminus U_R\cap U_B}|{\rm NEU}(v_i, G\setminus (H_i\cup G[U_R\cup U_B]))|\\
&&\\
&+ & \displaystyle\sum_{v_i\in (U_R\cup U_B)\setminus U_R\cap U_B} |{\rm NEU}(H_i-(U_R\cup U_B\setminus\{v_i\}))|\\
 &&\\
  & \geq & |E_G(U_R\cap U_B, V(G)\setminus (U_R\cup U_B)|+ \\
   &&\\
&+ & \displaystyle\sum_{v\in (U_R\cup U_B)\setminus (U_R\cap U_B)}({{\rm deg}_G(v)\over 2}+{2(r+1)\over \delta}). 
\end{array}
$$ 
Define
$$\begin{array}{lllll}
J & = & \displaystyle\sum_{v\in (U_R\cup U_B)\setminus (U_R\cap U_B)}\left({{\rm deg}_G(v)\over 2}+{2(r+1)\over \delta}\right) & \\
 &&&&\\
& = & {2(r+1)\over \delta}|(U_R\cup U_B)\setminus (U_R\cap U_B)| & + &
\displaystyle\sum_{v\in (U_R\cup U_B)\setminus (U_R\cap U_B)}{{\rm deg}_G(v)\over 2}.
\end{array}$$
Assume that $(U_R\cup U_B)\setminus (U_R\cap U_B)=\{u_1,u_2,\ldots,u_d\}$
such that ${\rm deg}_G(u_1)\leq {\rm deg}_G(u_2)\leq\cdots\leq{\rm deg}_G(u_d)$.
Note that $d\geq 2r+2-2|U_R\cap U_B|$. 
Set 
$$d'=\max\{0, r+1-|U_R\cap U_B|\},\ Y=\{u_1,u_2,\ldots,u_{d'}\},\  Z=Y\cup (U_R\cap U_B),$$
where if $d'=0$, then $Y=\varnothing$. Note that $|Z|\geq r+1$. 
Now we have
$$\begin{array}{lll}
|{\rm NEU}(G)| & \geq &  |E_G(U_R\cap U_B, V(G)\setminus (U_R\cup U_B)|+J\\
&&\\
& \geq &  |E_G(U_R\cap U_B, V(G)\setminus (U_R\cup U_B)|+{2(r+1)\over \delta}|(U_R\cup U_B)\setminus (U_R\cap U_B)|\\
&&\\
& +&  \displaystyle\sum_{u\in (U_R\cup U_B)\setminus (U_R\cap U_B)}{{\rm deg}_G(u)\over 2}
\end{array}$$
Note that $|U_R\cap U_B|\leq {2(r+1)\over \delta}$. Hence, 
$$\begin{array}{lll}
|{\rm NEU}(G)| & \geq & |E_G(U_R\cap U_B, V(G)\setminus(U_R\cap U_B)|+
\displaystyle\sum_{u\in (U_R\cup U_B)\setminus (U_R\cap U_B)}{{\rm deg}_G(u)\over 2}\\
&&\\
 & \geq &
|E_G(U_R\cap U_B, V(G)\setminus(U_R\cap U_B)|+
\displaystyle\sum_{u\in Y}{\rm deg}_G(u)\\
&&\\
& \geq &
|E_G(Z, V(G)\setminus Z)| \\
&&\\
& \geq  & {\rm cut}_{r+1}(G),  
\end{array}$$
which is a contradiction. 

Therefore, the claim follows, i.e., $G^R$ or $G^B$ 
has a connected component  $P$ with $n-q$ vertices where $0\leq q\leq r$.
Without loss of generality, suppose that $P$ is such a component in $G^R$ with $n-q$ vertices.\\

\noindent{\bf Step V (Claim):} $P$ contains a ${\cal G}$-tree with $n-r$ vertices.\\

If $q=0$, then $P=G^R$ is a ${\cal G}$-subgraph. Also, 
the length of the alternating $2$-coloring  is at least 
$|E(G)|-|E(G_1)|+1$. Now since $G$ has 
${\cal G}$-forest property, $G^R$ has a ${\cal G}$-forest. 
Also, note that 

$$
\begin{array}{ll}
|E(G^R)|\geq {|E(G)|-{\rm cut}_{r+1}(G,{\cal G})\over 2} &> {5\over 24}n^2-{1\over 2}(r+1)(n-r-1)\\
&\\
& > {k(3k-2)(n-2)\over 2}+ (r+k-1)(n-1), 
\end{array}
$$
provided that $N$ is large enough.
Therefore, in view of Lemma~\ref{treesubgraph}, the graph $G^R$ has a ${\cal G}$-tree with $n-r$ vertices, which implies the main claim. 

Thus suppose $q\not = 0$.  
Now we show that for large enough $N$ and for any $i\in\{1,2,\ldots,k\}$, $|E(P)\cap E(G_i)|\geq {2k-3\choose 2}+1$.
On the contrary, suppose that there is an $i_0\in\{1,2,\ldots,k\}$ such that $|E(P)\cap E(G_{i_0})|\leq {2k-3\choose 2}$. 
Therefore, there are at most ${q\choose 2}+{2k-3\choose 2}$ red edges in $G_{i_0}$, and consequently, since the edges of $G_{i_0}$ are consecutive in $\sigma$, there are 
at most ${q\choose 2}+{2k-3\choose 2}+1$ blue edges in $G_{i_0}$. 
This implies 
\begin{equation}\label{countneuGi}
|{\rm NEU}(G_{i_0})| \geq |E(G_{i_0})|- 2{q\choose 2}-2{2k-3\choose 2}-1.
\end{equation}
If $i_0=k$ and $N$ is large enough, then the number of neutral edges is at least 
$$|E(G_{k})|- 2{q\choose 2}-2{2k-3\choose 2}-1>{5\over 12}n^2-2{r\choose 2}-2{2k-3\choose 2}-1> (r+1)(n-r-1),$$ 
which is a contradiction. 
Thus we suppose $i_0<k$.  In view of (\ref{degGk}), 
the number of neutral edges incident with the vertices of $V(G)\setminus V(P)$ in the graph $G_k$
is at least 

$$
\begin{array}{l}
\displaystyle\sum_{u\in V(G)\setminus V(P)} ({\rm deg}_{G_k}(u)-{\rm deg}_{G_k^R}(u)-{\rm deg}_{G_k^B}(u)) \geq \\
\\
\displaystyle\sum_{u\in V(G)\setminus V(P)} \left({\rm deg}_{G_k}(u)-(q-1)-{{\rm deg}_{G_k}(u)+7\over 2}\right),
\end{array}$$ 
and consequently,
\begin{equation}\label{countneuGk}
|{\rm NEU}(G_k)| \geq \displaystyle\sum_{u\in V(G)\setminus V(P)} ({{\rm deg}_{G_k}(u)\over 2}-q-{5\over 2}). 
\end{equation}
Hence, in view of (\ref{countneuGi}) and (\ref{countneuGk}), the number of neutral edges in $G$ is at least 
$$
\begin{array}{lcl}
|{\rm NEU}(G_{i_0})|+|{\rm NEU}(G_k)| & \geq & |E(G_{i_0})|- 2{q\choose 2}-2{2k-3\choose 2}-1\\
&& \\
&& +\displaystyle\sum_{u\in V(G)\setminus V(P)} ({{\rm deg}_{G_k}(u)\over 2}-q-{5\over 2})\\
&& \\
&\geq & |E(G_{i_0})|+ {1\over 2} \delta nq-2q^2-2{2k-3\choose 2}-{3\over 2}q-1. 
\end{array}$$ 
Since $1\leq q\leq r=o(n)$ and $k=o(\sqrt{n})$, for large enough $N$, the number of neutral edges is more than 
$|E(G_{i_0})|-1$, which is impossible. Hence, for any $1\leq i\leq k$, $|E(P)\cap E(G_i)|\geq {2k-3\choose 2}+1$.  

Suppose 
that $F$ is a spanning ${\cal G}$-subgraph of $P$ with $k$ edges and the minimum number of cycles. 
In what follows, we show that 
$F$ is a forest with $k$ edges and in view of Lemma~\ref{treesubgraph}, we extend it to a ${\cal G}$-tree of $P$ with $n-r$ vertices.  On the contrary, suppose that $F$ is~not a forest.  Let 
$e$ be an edge of a cycle of $F$ and that $e\in E(F)\cap E(G_{j})$. 
Since $|E(P)\cap E(G_i)|\geq {2k-3\choose 2}+1$ for $i=1,2,\ldots,k$, and also, 
since $F$ has at most $2k-3$ non-isolated vertices, there exists a red edge $e'\in E(G_{j})$ incident with some isolated vertex of $F$. 
Define 
$F'=(F- e)\cup \{e'\}$. Clearly, $F'$ is a ${\cal G}$-subgraph of $P$ with $k$ edges. Also, the number of cycles of $F'$ is less than that of $F$,  which is a contradiction. 

In view of ${\rm cut}_{r+1}(G,{\cal G})\leq (r+1)(n-r-1)$, one can see that 
for large enough $N$, the number of edges of $P$ is at least 
$$
\begin{array}{ll}
|E(G^R)|-{q\choose 2}\geq {|E(G)|-{\rm cut}_{r+1}(G,{\cal G})\over 2}-{q\choose 2} &> {5\over 24}n^2-{1\over 2}(r+1)(n-r-1)-{r\choose 2}\\
&\\
& > {k(3k-2)(n-2)\over 2}+ (r+k-1)(n-1). 
\end{array}
$$ 
Therefore, in view of Lemma~\ref{treesubgraph}, one can see that $P$ has a ${\cal G}$-tree subgraph with $n-r$ vertices. This implies the main claim, 
and consequently, 
${\rm ex}_{alt}(G,{\cal G},{\cal T}_{n-r})\leq {\rm ex}(G,{\cal G},{\cal T}_{n-r})$. In view of 
Lemma~\ref{ex}, the lemma follows. 
}\end{proof}
We are now in a position to prove Theorem~\ref{mainthm}. \\

\noindent{\bf Completing the proof of Theorem~\ref{mainthm}}. In view of Lemma~\ref{lastlem}, one can see that the graph $G$ has ${\cal G}$-forest property. Now by Lemma~\ref{main}, the assertion holds. 
$\hfill \blacksquare$\\

For an ordering $\sigma$ of the edge set of $G$, a 
{\it $\sigma$-tree} of $G$ is a tree subgraph of $G$ such that it has no two consecutive edges in the ordering $\sigma$.  
Define ${\cal T}_{n-r}^{\sigma, \Delta}$ to be the set of all $\sigma$-tree subgraphs of $G$ with $n-r$ vertices and maximum degree at most 
$\Delta$. It is worth noting that, in view of the proof of 
Lemma~\ref{main}, we can determind the chromatic number of 
the graph ${\rm KG}(G, {\cal G}, {\cal T}_{n-r}^{\sigma, {n\over 2}+{3k+4\over 2}})$, where $G$ and ${\cal G}$ meet the conditions of the lemma 
and $\sigma$ is the ordering introduced in the first step. 
Precisely,   
$$\chi({\rm KG}(G, {\cal G}, {\cal T}_{n-r}^{\sigma, {n\over 2}+{3k+4\over 2}}))={\rm cut}_{r+1}(G, {\cal G}).$$
Note that any tree subgraph of $G^R$ or $G^B$ is a $\sigma$-tree subgraph, and moreover, in view of (\ref{degree}), it has maximum degree at most ${n\over 2}+{3k+4\over 2}$. 

By the aforementioned results, one can introduce some graph whose chromatic number is about twice of its clique number. To see this, 
note that if $n$ is sufficiently large, then $\chi({\rm KG}(K_n,{\cal T}_n))=n-1$. Also, 
if $n$ is an even integer,
then the complete graph $K_n$ is decomposable into edge-disjoint Hamiltonian paths. Hence, the clique number of
${\rm KG}(K_n,{\cal T}_n)$ is $\lfloor {n\over 2}\rfloor$.  It would be of interest to note that we cannot drop some conditions in 
Lemma~\ref{main}. For instance, one can see that Lemma~\ref{main} 
does~not hold for sparse graphs. To see this, consider the general Kneser graph ${\rm KG}(G,{\cal T}_n)$, where $G$ is a connected 
graph with $n$ vertices and $|E(G)|\leq 2n-3$.  One can check that 
the chromatic number of this graph is equal to $1$, while the size of the minimum cut of $G$ can be equal to $3$. 
Finally, in~\cite{2014arXiv1401.0138A}, it was shown that 
for a large family of graphs, we have 
$$\chi({\rm KG}(G,{\cal T}_3))< {\rm cut}_{n-2}(G)=|E(G)|-{\rm ex}(G, {\cal T}_{3}).$$ 
Hence, it seems that Lemma~\ref{main} does~not hold for the family of small trees.\\ 

Hedetniemi's conjecture asserts that
the chromatic number of the Categorical product of two graphs is the minimum of
that of graphs. This conjecture has been studied in the literature, see~\cite{2014arXiv1403.4404A, MR2171371, MR2445666, MR2825542}. 
A family of graphs is tight if  Hedetniemi's conjecture holds for any two graphs of this family. 
By Theorem~\ref{mainthm} and  in view of the results of~\cite{2014arXiv1403.4404A}, we can enrich 
a family of tight graphs.\\ 

\noindent{\bf Acknowledgement:}
The research of Meysam Alishahi was in part supported by a grant from IPM (No. 94050014).  
Also, the research of Hossein Hajiabolhassan was supported by ERC advanced grant GRACOL. Furthermore, 
part of this work done during a visit of  Hossein Hajiabolhassan to the Mittag-Leffler Institute (Djursholm, Sweden). 
The authors are grateful to Professor Carsten~Thomassen for his fruitful discussions. 
Moreover, they would like to thank Skype for sponsoring their endless conversations in two countries. 
\def\cprime{$'$} \def\cprime{$'$}

\end{document}